\documentclass[12pt]{article}
\usepackage{jeffe}
\usepackage{verbatim}
\def\cyc{\mathcal{C}}
\usepackage{fullpage}
\newtheorem{theorem}{Theorem}
\newtheorem{lemma}[theorem]{Lemma}
\newtheorem{fact}[theorem]{Fact}
\newtheorem{corollary}[theorem]{Corollary}

\begin{document}
\title{Antimagic labelings of regular bipartite graphs: \\An application of the Marriage Theorem}
\author{Daniel Cranston}
\maketitle
\abstract{
A labeling of a graph is a bijection from $E(G)$ to the set $\{1, 2,\ldots, |E(G)|\}$.  
A labeling is \textit{antimagic} if for any distinct vertices $u$ and $v$, the sum of the labels on edges incident to $u$ is different from the sum of the labels on edges incident to $v$.  We say a graph is antimagic if it has an antimagic labeling.
In 1990, Ringel conjectured that every connected graph other than $K_2$ is antimagic.  
In this paper, we show that every regular bipartite graph (with degree at least 2) is antimagic.  Our technique relies heavily on the Marriage Theorem.
}
\section{Introduction}

In this paper, we study a problem of edge-labeling.  For convenience,
we formally define a \textit{labeling} of a graph $G$ to be a bijection from
$E(G)$ to the set $\{1,\ldots,|E(G)|\}$.  A \textit{vertex-sum} for a labeling
is the sum of the labels on edges incident to a vertex $v$; we also call this
the \textit{sum at $v$}.  A labeling is \textit{antimagic} if the vertex-sums
are pairwise distinct.  A graph is \textit{antimagic} if it has an antimagic
labeling.

Hartsfield and Ringel~\cite{hartsfield} introduced antimagic labelings in 1990
and conjectured that every connected graph other than $K_2$ is antimagic.
The most significant progress on this problem is a result of Alon, Kaplan, Lev,
Roditty, and Yuster~\cite{alon}, which states the existence of a constant $c$
such that if $G$ is an $n$-vertex graph with $\delta(G)\ge c\log n)$, then $G$
is antimagic.  Large degrees satisfy a natural intuition: the more edges are
present, the more flexibility there is to arrange the labels and possibly
obtain an antimagic labeling.

Alon et al.~also proved that $G$ is antimagic when $\Delta(G)\ge |V(G)|-2$,
and they proved that all complete multipartite graphs (other than $K_2$) are
antimagic.  Hartsfield and Ringel proved that paths, cycles, wheels, and
complete graphs are antimagic.  

In this chapter, we show that every regular bipartite graph (with degree at
least 2) is antimagic.  Our proof relies heavily on the Marriage Theorem,
which states that every regular bipartite graph has a 1-factor. 
By induction on the vertex degree, it follows that a regular bipartite graph
decomposes into 1-factors.  Recall that a $k$-factor is a $k$-regular spanning
subgraph, so the union of any $k$ $1$-factors is a $k$-factor.  Throughout this
paper, we refer to the partite sets of the given bipartite graph as $A$ and
$B$, each having size $n$.

With respect to a given labeling, two vertices \textit{conflict} if they
have the same sum.  We view the process of constructing an antimagic labeling
as resolving the ``potential conflict'' for every pair of vertices.  We will
label the edges in phases.  When we have labeled a subset of the edges, we call
the resulting sum at each vertex a \textit{partial sum}.

Our general approach is to label all but a single 1-factor so that the
partial sums in $A$ are multiples of 3, while the partial sums in $B$ are
non-multiples of 3.  At this stage no vertex of $A$ conflicts with a vertex
of $B$.  We then label the final 1-factor with reserved labels that are
multiples of 3 so that we resolve all potential conflicts within $A$ and
within $B$.  Before we begin the general approach, we observe two facts that
together show that 2-regular graphs are antimagic.

\begin{fact}
\label{fact1}\cite{hartsfield}
Every cycle is antimagic.
\end{fact}
\begin{proof}
Assign the labels to edges as $1,3,\ldots,n,n-1,\ldots,4,2$ in order around
an $n$-cycle (if $n$ is odd; otherwise, $n$ and $n-1$ are switched in the
middle).  The sums are $4,8,\ldots,10,6,3$; that is, the sums of consecutive
odd integers are even multiples of $2$, while the sums of consecutive even
integers are odd multiples of $2$.
\end{proof}

\begin{fact}
\label{fact2}
If $G_1$ and $G_2$ are each regular antimagic graphs, then the disjoint
union of $G_1$ and $G_2$ is also antimagic.
\end{fact}
\begin{proof}
Index $G_1$ and $G_2$ so that vertices in $G_2$ have degree at least as large
as those in $G_1$.  Let $m_1=|E(G_1)|$.  Place an antimagic labeling on $G_1$,
using the first $m_1$ labels.  Label $G_2$ by adding $m_1$ to each label in an
antimagic labeling of $G_2$.  Let $k$ be the degree of the vertices in $G_2$.

Translating edge labels by $m_1$ adds $m_1k$ to the sum at each vertex of
$G_2$, so the new labeling of $G_2$ has distinct vertex sums.  Hence there are
no conflicts within $G_1$ and no conflicts within $G_2$.  There are also no
conflicts between a vertex in $G_1$ and one in $G_2$, since each vertex-sum in
$G_1$ is less than $m_1k$ and each vertex-sum in $G_2$ is greater than $m_1k$.
\end{proof}

More generally, given any labeling of a regular graph, adding the same amount
to each label does not change the pairs of vertices that conflict.
Fact~\ref{fact1} and Fact~\ref{fact2} immediately yield:

\begin{corollary}
Every simple 2-regular graph is antimagic.
\end{corollary}

For degrees larger than 2, we will consider odd and even degree separately.  Although 2-regular graphs
are easy, the general construction is a bit more complicated for even degree
than for odd degree.

%
\section{Regular bipartite graphs with odd degree}

We have observed that a $k$-regular bipartite graph $G$ decomposes into
1-factors.  We can combine these 1-factors in any desired fashion.  In
particular, when $k$ is odd and at least $5$, we can decompose $G$ into
a $(2l+2)$-factor and a $3$-factor, where $l\ge0$.  Our aim will be to combine
special labelings of these two factors to obtain an antimagic labeling of $G$.
The case $k=3$ is handled separately; we do this before the general argument.

\begin{theorem}\label{3reg}
Every 3-regular bipartite graph is antimagic.
\end{theorem}
\begin{proof}
Since $G$ has $3n$ edges, we have the same number of labels in each congruence
class modulo 3.  For convenience, we use the term \textit{$j$-labels} to
designate the first $n$ positive integers that are congruent to $j$ modulo 3,
where $j\in\{0,1,2\}$.

Decompose $G$ into a 1-factor $H_1$ and a 2-factor $H_2$.  We will reserve the
$0$-labels for $H_1$.  We will label $H_2$ with the $1$-labels and $2$-labels
so that the partial sum at each vertex of $A$ is $3n$.  We do this by
pairing each $1$-label $i$ with the $2$-label $3n-i$.  These pairs have sum
$3n$; at each vertex of $A$, we use the two integers in some pair.
Subsequently, every assignment of $0$-labels to $H_1$ yields distinct
vertex-sums within $A$.

We have assigned a pair of labels at each vertex of $A$ in $H_2$, but we have
not decided which edge gets which label.  Next we try to make this choice
so that in $H_2$ the partial sums at vertices of $B$ will not be multiples of
3.  In each component of $H_2$, we will fail at most once.

Let $\cyc$ be a cycle that is a component of $H_2$.  We have a 1-label and
a 2-label at each vertex of $A$.  As we follow $\cyc$, if we have a 1-label
and then a 2-label at a vertex of $A$, then the next vertex of $A$ should have
a 2-label followed by a 1-label (and vice versa), since the sum of two 1-labels
or two 2-labels is not a multiple of 3.  If $|V(\cyc)\cap A|$ is even, then
we succeed throughout; if $|V(\cyc)\cap A|$ is odd, then at one vertex of
$\cyc$ in $B$ we will have a 1-label and a 2-label.  Call such a vertex of
$B$ \textit{bad}.  A cycle in $H_2$ has a bad vertex only if it has length
at least 6, so at most $n/3$ vertices in $B$ will be bad.  Let $m$ be the
number of bad vertices.

To avoid conflicts between vertices of $A$ and bad vertices of $B$, we will
make the vertex-sum at each bad vertex smaller than at any vertex of $A$.
Furthermore, we will make the partial sums in $H_2$ at these vertices equal.
Consider the 1-labels and 2-labels from 1 through $3m-1$; group them into
pairs $j$ and $3m-j$.  The sum in each such pair is $3m$, which is at most $n$.
Allocate the pairs for $H_2$ to vertices of $A$ so that at each bad vertex of
$B$, the labels are the small elements from pairs in the original pairing
and form a pair with sum $3m$ in this most recent pairing.

Now we need to label $H_1$.  We must achieve three goals: resolve all conflicts
among the good vertices in $B$, resolve all conflicts among the bad vertices in
$B$, and resolve all conflicts between $A$ and the bad vertices in $B$.

We consider the last goal first.  For every assignment of 0-labels to $H_1$,
the vertex-sums in $A$ will be $\{3n+3, 3n+6, \ldots, 6n-3, 6n\}$.  To ensure
that the vertex-sums at the bad vertices in $B$ will be less than $3n+3$, we
use the smallest 0-labels at the bad vertices.  Since there are at most $n/3$
bad vertices, every 0-label at such a vertex is at most $n$.  Thus, every sum
at a bad vertex is at most $2n$, which is less than $3n$.  Furthermore, the
sums at bad vertices are $3m$ plus distinct 0-labels; hence they are distinct,
which completes the second goal.

For the first goal, let $b_1, b_2, b_3,\ldots$ denote the good vertices of
$B$ in order of increasing partial sum from $H_2$ (there may be ties).  We
assign the remaining 0-labels to edges of $H_1$ at $b_1,b_2,\ldots$ in
increasing order.  Since the 0-labels are distinct, this prevents conflicts
among the good vertices in $B$.
\end{proof}

For larger even degree, we will construct an antimagic labeling from
special labelings of two subgraphs.  Like the labeling we constructed for
$3$-regular graphs, the first labeling will have equal sums at vertices of
$A$, but this time we guarantee that all sums at vertices of $B$ are not
congruent modulo 3 to the sums at vertices of $A$.

\begin{lemma}
If $G$ is a $(2l+2)$-regular bipartite graph with parts $A$ and $B$ of size $n$,
then $G$ has a labeling such that the sum at each vertex of $A$ is some fixed
value $t$ and the sum at each vertex of $B$ is not congruent to $t$ modulo 3.
\label{key lemma}
\end{lemma}

\begin{proof}
As remarked earlier, we can decompose $G$ into a $2l$-factor $H_{2l}$ and
a 2-factor $H_2$.  Let $m=(2l+2)n$; thus $m$ is the largest label.  Since $m$
is even, we can partition the labels 1 through $m$ into pairs that sum to $m+1$.
With $m+1\equiv 2a(\bmod 3)$, each pair consists of two elements in the same
congruence class as $a$ modulo 3 or elements in the two other congruence
classes modulo 3.  Call these \textit{like-pairs} and \textit{split-pairs},
respectively.

At each vertex of $A$, we will use $l$ of these pairs as labels in $H_{2l}$.
Thus each vertex of $A$ will have partial sum $(m+1)l$ in $H_{2l}$; we will assign
the pairs so that the partial sums in $B$ are not congruent to $(m+1)l$ modulo 3.
We use the pairs in which the smaller label ranges from 1 to $ln$.  Note that
$H_{2l}$ decomposes into even cycles (for example, we can take $2l$ 1-factors
two at a time to generate 2-factors whose union is $H_{2l}$).

For each cycle in the decomposition of $H_{2l}$ into even cycles, at vertices
of $A$ we use pairs of labels of the same type: all like-pairs or all
split-pairs.  When using split-pairs, we assign the labels so that the same
congruence class modulo 3 is always first.  If we have all like-pairs or
all split-pairs, this ensures that at each vertex of $B$, each cycle
contributes an amount to the sum that is congruent to $2a$ modulo 3.
There is at most one cycle where we are forced to use both like-pairs and
split-pairs.  Let $x$ and $y$ be the vertices of $B$ where, in this cycle,
we switch between like-pairs and split-pairs.  At each vertex of $A$, the
partial sum in $H_{2l}$ is $(m+1)l$.  At each vertex of $B$, except $x$ and $y$,
the partial sum is congruent to $(m+1)l$ modulo 3.

On $H_2$, we use the remaining pairs of labels so that we add $m+1$ to each
partial sum in $A$, but what we add to each partial sum in $B$ is not congruent
to $m+1$ modulo 3.  If we can do this (and treat $x$ and $y$ specially), then
the sum at each vertex of $A$ will be $(m+1)(l+1)$, while at each vertex of
$B$ the sum will be in a different congruence class modulo 3 from $(m+1)(l+1)$.

On each cycle, we use the pairs of labels that contain the smallest unused
labels.  Thus, every third pair we use is a like-pair; the others are split
pairs.  We begin with a like-pair and alternate using a like-pair and a
split-pair until the like-pairs allotted to that cycle are exhausted.  For the
remaining split-pairs, we alternate them in the form $(a+1,a+2)$ followed by
$(a+2,a+1)$; in this way the sum of the two labels used at any vertex of $B$
is not congruent to $2a$ modulo 3.  If no like-pair is available to be used on
the cycle, then the cycle has length 4 and we label it with split-pairs in
the form $(a+1,a+2),(a+2,a+1)$, and the same property holds.

One or two cycles in $H_2$ may contain the vertices $x$ and $y$, where the
sum in $H_{2l}$ differs by 1 from a value congruent to $(m+1)l$ modulo 3.
Suppose that the sums in $H_{2l}$ at $x$ and $y$ are $(m+1)l+t_1$ and $(m+1)l+t_2$.  We want the sum at $x$ in $H_2$ to be either $2a-t_1+1 (\bmod 3)$ or $2a-t_1+2$.  Similary, we want the sum at $y$ in $H_2$ to be in $\{2a-t_2+1, 2a-t_2+2\}$.  The more difficult case is when $x$ and $y$ lie on the same cycle in $H_2$.  However, given the realization that we have two choices each for the sums (modulo 3) at  $x$ and $y$, it is not difficult to adapt the labeling given above for cycles of $H_2$ so that it applies in the current case as well.

At these vertices we want the contribution from $H_2$ to be congruent to
$2a$ modulo 3.  We deal with these first and can then make the argument
above for the remaining cycles.  If $x$ and $y$ lie on a single 4-cycle,
then we use two like-pairs or two split-pairs ordered as $(a+1,a+2),(a+1,a+2)$.
If one or both of $x$ and $y$ lie on a longer
cycle, then at each we put edges from two like-pairs or from two split-pairs
ordered as $(a+1,a+2),(a+1,a+2)$.  The remaining pairs, whether they are
like-pairs or split-pairs as we allocate them to this cycle, can be filled in
so that like-pairs are not consecutive anywhere else and neighboring
split-pairs alternate their ``orientation''.

Thus the labeling of $H_2$ enables us to keep the overall sum at each vertex
of $B$ out of the congruence class of $(m+1)(l+1)$ modulo 3.
\end{proof}

\begin{lemma}
If $G$ is a 3-regular bipartite graph with parts $A$ and $B$, where
$B=\{b_1,\ldots,b_n\}$,  then $G$ has a labeling so that at each $b_i$ the
sum is $3n + 3i$, and for each $i$ exactly one vertex in $A$ has sum $3n + 3i$.
\label{3factor-neq}
\end{lemma}
\begin{proof}
Decompose $G$ into three 1-factors: $R$, $S$, and $T$.  In $R$, use label
$3i-2$ on the edge incident to $b_i$; let $a_i$ be the other endpoint of this
edge.  In $S$, use label $3n+3-3i$ on the edge incident to $a_i$; call the
other endpoint of this edge $b'_i$.  In $T$, use label $3i-1$ on the edge
incident to $b'_i$; call the other endpoint of this edge $a'_i$.
Note that each 1-factor received the labels from one congruence class modulo 3.

The partial sum in $S\cup T$ at each vertex of $B$ is $3n+2$.  Hence, the
sum at $b_i$ for all of $G$ is $3n+3i$.  Similarly, the partial sum in
$R\cup S$ at each vertex of $A$ is $3n+1$.  Hence, the vertex-sum at $a'_i$ is
$3n+3i$.
\end{proof}

\begin{theorem}
Every regular bipartite graph of odd degree is antimagic.
\end{theorem}

\begin{proof}
Let $G$ be a regular bipartite graph of degree $k$.  Theorem~\ref{3reg} is the
case $k=3$.  For $k>3$, let $k=2l+5$ with $l\ge0$, and decompose the graph $G$
into a 3-factor $H_3$ and a $(2l+2)$-factor $H_{2l+2}$. Label $H_{2l+2}$ as in
Lemma~\ref{key lemma}; this uses labels 1 through $(2l+2)n$.  Add $3n$ to each
label, leaving labels 1 through $3n$ for $H_3$.  Each vertex-sum increases by
$9n$, which is a multiple of 3, so the congruence properties obtained in
Lemma~\ref{key lemma} remain true for the new labeling.

Let $b_i$ denote the vertices of $B$ in order of increasing partial sum in
$H_{2l+2}$.  Label $H_3$ as in Lemma~\ref{3factor-neq}.  Because all the partial sums
in $H$ are multiples of 3, the labeling of $H_{2l+2}$ resolves each potential
conflict between a vertex of $A$ and a vertex of $B$.  Because the $b_i$ are in
order of increasing partial sum in $H_{2l+2}$, the labeling of $H_3$ resolves all
potential conflicts within $B$.  Similarly, since the labeling of $H_{2l+2}$ gives
the same partial sum to all vertices of $A$, the labeling of $H_3$ resolves
all potential conflicts within $A$.

We have checked that the labeling is antimagic.
\end{proof}

\section{Regular bipartite graphs with even degree}
\begin{lemma}
Let $n$ be a positive integer. 
If $n$ is even, then we can partition $\{1,2,\ldots,3n\}$ into triples such that the sum of each triple is $6n+3$ or $3n$.  
If $n$ is odd, then we can partition $\{1,2,\ldots,3n\}$ into triples such that the sum of each triple is $6n$ or $3n$.  
Furthermore, each triple consists of one integer from each residue class modulo 3.
\label{3factor-eq}
\end{lemma}
\begin{proof}
Suppose $n$ is even.  We partition the labels into triples so that the sum of each triple is either $3n$ or $6n+3$.  Consider the triples $(3n-3i+3, 3n-3i+2, 6i-2)$ and $(3i, 3i-1, 3n-6i+1)$ for $1\leq i \leq n/2$.  
Triples of the first type sum to $6n+3$ and triples of the second type sum to $3n$.  

Suppose $n$ is odd.  We partition the labels into triples so that the sum of each triple is either $3n$ or $6n$.  Consider the triples $(3n-3i+3,3n-3i+2,6i-5)$ for $1\leq i \leq \ceil{n/2}$ and $(3i, 3i-1, 3n-6i+1)$ for $1\leq i\leq\floor{n/2}$.
Triples of the first type sum to $6n$ and triples of the second type sum to $3n$.  
\end{proof}

\begin{theorem}
Every regular bipartite graph of even degree at least 8 is antimagic.
\label{8reg}
\end{theorem}
\begin{proof}
We decompose $G$ into two 3-factors and a $(2l+2)$-factor; call these $G_3, H_3$, and $H_{2l+2}$, respectively.
We label $H_{2l+2}$ as in Lemma~\ref{key lemma}, using all but the $6n$ smallest labels.  
This resolves every conflict between a vertex of $A$ and a vertex of $B$.

We partition the labels $\{3n+1,3n+2,\ldots,6n\}$ into triples as in Lemma~\ref{3factor-eq}.
In $G_3$, at each vertex of $A$ we will use the the three labels of some triple.  To ensure the sum at each vertex of $B$ is $0(\bmod~3)$, we do the following.  Partition the 3-factor into three 1-factors; We use 0-labels on the first 1-factor, 1-labels on the second 1-factor, and 2-labels on the third 1-factor.  

Now consider the partial sums in the union of $H_{2l+2}$ and $G_3$; let $b_i$ denote the vertices of $B$ in order of increasing partial sum.
Label $H_3$ as in Lemma~\ref{3factor-neq}.
This resolves every conflict between two vertices in the same part.
Hence, the labeling is antimagic.
\end{proof}

Lemma~\ref{triples} is very similar to Lemma~\ref{3factor-eq}.  Lemma~\ref{triples} serves the same role in the proof of Theorem~\ref{6reg} that Lemma~\ref{3factor-eq} does in the proof of Theorem~\ref{8reg}.

\begin{lemma}
Let $n$ be a positive integer.
Let $H$ be the set of positive labels less than $4n$ that are not 0 modulo 4,
i.e. $H=\{1,2,3,5,6,\ldots,4n-2,4n-1\}$.
If $n$ is even, then we can partition $H$ into triples such that the sum of each triple is either $4n-2$ or $8n+2$.
If $n$ is odd, then we can partition $H$ into triples such that the sum of each triple is either $4n-2$ or $8n-2$.
Furthermore, each triple consists of one integer from each nonzero residue class modulo 4.
\label{triples}
\end{lemma}
\begin{proof}
Suppose $n$ is even.
We have triples of the form $(8i-3,4n-4i+2,4n-4i+3)$, 
with $1\leq i\leq n/2$, and triples of the form $(4n-8i+1,4i-2,4i-1)$, 
with $1\leq i\leq n/2$.
It is easy to see that triples of the first form sum to $8n+2$ and that triples of the second form sum to $4n-2$.  It is straightforwad to verify that these triples partition $H$.

Suppose $n$ is odd.
We have triples of the form $(8i-7,4n-4i+2,4n-4i+3)$, with $1\leq i\leq \ceil{n/2}$, and triples of the form $(4n-8i+1,4i-2,4i-1)$, with $1\leq i\leq \floor{n/2}$.  
It is easy to see that triples of the first form sum to $8n-2$ and that triples of the second form sum to $4n-2$.  It is straightforwad to verify that these triples partition $H$.
\end{proof}

\begin{theorem}
Every 6-regular bipartite graph is antimagic.
\label{6reg}
\end{theorem}
\begin{proof}
Throughout this proof we assume that $n$ is odd.  The argument is analagous when $n$ is even, so we omit the details.
We decompose $G$ into a 1-factor, a 2-factor, and a 3-factor; call these $H_1$, $H_2$, and $H_3$, respectively. We label $H_3$ with the labels that are less than $4n$ and are not 0 modulo 4, so that the partial sum at each vertex of $B$ is $2(\bmod~4)$ and the partial sum at each vertex of $A$ is $4n-2$ or $8n-2$.
To do this we partition the labels for $H_3$ into triples as specified in Lemma~\ref{triples}.

At each vertex of $A$, we use the three labels in some triple.  More exactly, we decompose $H_3$ into three 1-factors; we use $1(\bmod~4)$ labels on the first 1-factor, use $2(\bmod~4)$ labels on the second 1-factor, and use $3(\bmod~4)$ labels on the third 1-factor.  This ensures that the partial sum at each vertex of $B$ is $2(\bmod~4)$.

We label $H_2$ with the labels $4n+1$ through $6n$, so that the partial
sum at each vertex of $A$ is $10n+1$ and the sum at each vertex of $B$ is $\not\equiv 10n+1(\bmod~4)$.
To do this, we partition the labels for $H_2$ into pairs that sum to $10n+1$.  We consider the labels in each pair modulo 4.  
We have two types of pairs: $(1,2)$ pairs and $(3,0)$ pairs (since $n$ is odd).  

We want to avoid using two labels at a vertex of $B$ that sum to $3(\bmod~4)$.  
We choose the pairs of labels to use on each cycle arbitrarily, except that each cycle must use at least one $(1,2)$ pair and at least one $(3,0)$.  
We first use all the $(1,2)$ pairs, alternating them as $(1,2),(2,1),(1,2),(2,1),\ldots$, then use all the $(3,0)$ pairs, alternating them as $(3,0),(0,3),(3,0),(0,3),\ldots$.  As long as we use at least one $(1,2)$ pair and one $(3,0)$ pair on each cycle of $H_2$, we have no problems.  
Since we use at least one $(1,2)$ pair and one $(3,0)$ pair on each cycle of $H_2$, we are able to avoid vertex sums in $B$ that are congruent to $3(\bmod~4)$.

Now we consider partial sums in $H_2\cup H_3$.  The partial sum at each vertex of $A$ is $14n-1$ or $18n-1$.  The partial sum at each vertex of $B$ is not congruent to 3 modulo 4.  The labels we will use on the $H_1$ are all multiples of 4.  Hence, regardless of how we label $H_1$, no vertex in $A$ will conflict with any vertex in $B$.  We call a vertex in $A$ \textit{small} if it's partial sum in $H_2\cup H_3$ is $14n-1$; otherwise, we call it \textit{big}.  It is clear that regardless of how we label $H_1$, no big vertex will conflict with another big vertex; similarly, no small vertex will conflict with a small vertex.  Observe that the largest possible sum at a small vertex is $14n-1 + 4n = 18n -1$.  The smallest possible sum at a big vertex is $18n-1+4 = 18n+3$.  Hence, no small vertex will conflict with a big vertex.  Thus, we choose the labels for $H_1$ to ensure that no two vertices in $B$ conflict.

Let $b_i$ denote the vertices of $B$ in order of increasing partial sum in $H_2\cup H_3$. In $H_1$, we use label $4i$ at vertex $i$.
This ensures that vertex-sums in $B$ are distinct.
Thus, the labeling is antimagic.
\end{proof}

The proof for 4-regular graphs is more complicated than for 6-regular graphs.
In the 6-regular case, we labeled the 2-factor to ensure there were no conflicts between any vertex in $A$ and any vertex in $B$; we labeled the 1-factor and the 3-factor to ensure there were no conflicts between two vertices in the same part.
The proof for 4-regular graphs is similar, but since we have one less 2-factor,
we cannot ensure that all vertex-sums in $B$ differ modulo 4 from the vertex-sums in $A$.  So similar to the 3-regular graphs, we introduce \textit{good} and \textit{bad} vertices in $B$.  We handle bad vertices in a similar way to the case of the 3-regular graphs.

\begin{theorem}
Every 4-regular bipartite graph is antimagic.
\end{theorem}
\begin{proof}
Throughout this proof we assume that $n$ is odd.  The argument is analagous when $n$ is even, so we omit the details.
We decompose $G$ into a 1-factor $H_1$ and a 3-factor $H_3$.
We label $H_3$ with the labels that are less than $4n$ and are not 0 modulo 4, so that the partial sum at each vertex of $B$ is $4n-2$ or $8n-2$.  
To do this, we partition the labels for the 3-factor into triples as specified in Lemma~\ref{triples}.  At each vertex in $A$, we will use the three labels of a triple.
Consider a vertex of $B$: if its partial sum in the 2-factor is $2(\bmod~4)$, then we call the vertex \textit{bad}; otherwise, we call it \textit{good}.
We assign the labels of each triple to the edges at a vertex of $A$ to minimize the number of bad vertices in $B$.
Initially, we only assign to each edge a residue class: $1(\bmod~4)$, $2(\bmod~4)$, or $3(\bmod~4)$.  This determines which vertices in $B$ are bad.
We will then assign the labels to edges to minimize the largest partial sum at a bad vertex of $B$.  Since the bad vertices in $B$ will have vertex-sums in the same residue class (modulo 4) as the vertex-sums in $A$, to avoid conflicts we will ensure that the vertex-sum at every bad vertex is smaller than the smallest vertex-sum in $B$.

We begin by decomposing $H_3$ into three 1-factors.  We label each edge in the first 1-factor with a 1, each edge in the second 1-factor with a 2, and each edge in the third 1-factor with a 3.  However, this makes every vertex in $B$ bad.  To fix this, we consider the 2-factor labeled with 1s and 2s; specifically consider a single cycle in this 2-factor.  Select a vertex of $A$ on the cycle, then select every second vertex of $A$ along the cycle; at each of the selected vertices, swap the labels 1 and 2 on the incident edges.  If the cycle has length divisible by 4, then all of its vertices are now good.  If the length is not divisible by 4, then one bad vertex will remain.  Note that a cycle has a bad vertex only if its length is at least 6.  So, at most $n/3$ vertices are bad.
We now reduce the number of bad vertices further, as follows.

If a vertex is bad, consider the incident edge labeled 3, and the edge labeled 2 that is adjacent in $A$ to this first edge; these two edges form a \textit{bad path}.  We will swap the two labels on a bad path to reduce the number of bad vertices.
Consider the graph induced by bad paths; each component is a path or a cycle.  In a path component, we swap the labels on every second bad path; this fixes all the bad vertices.  We handle cycle components similarly, although in each cycle one bad vertex may remain (similar to the previous step).  
Thus, after this step, at most $1/3$ of the previously bad vertices remain bad.  So, at most $n/9$ vertices remain bad.
We also need to verify that when we swap the labels on a bad path, no good vertex becomes bad.

If a good vertex has partial sum $3(\bmod~4)$, we call it \textit{heavy}; if it has partial sum $1(\bmod~4)$, we call it \textit{light}.  Before we swap the labels on any bad path, the triple of labels incident to a light vertex is $(1,1,3)$; the triple incident to a heavy vertex is $(2,2,3)$.  Thus, we do not swap any labels incident to a light vertex.  However, the labels incident to a heavy vertex could become $(2,3,3)$ or even $(3,3,3)$.  In each case though, the vertex remains good.

Finally, if any vertex in $A$ is adjacent to two or more bad vertices, we swap the labels on its incident edges to make each vertex good.  Thus, we have at most $n/9$ bad vertices and each vertex in $A$ is adjacent to at most one bad vertex.  Now we assign the actual labels to the edges (rather than only the residue classes) so that the partial sum at each bad vertex is small.  We assign the $n/9$ smallest $1(\bmod~4)$ labels to be incident to the bad vertices; the largest is less than $4n/9$.  Similary, we assign the $n/9$ smallest $2(\bmod~4)$ labels to be incident to the bad vertices; again the largest is less than $4n/9$.  Each time we assign a label, we also assign the other labels in its triple.  Since each $2(\bmod~4)$ label is in a triple with the $3(\bmod~4)$ label one greater, the $n/9$ smallest $3(\bmod~4)$ labels are already assigned.  So we assign the next $n/9$ smallest $3(\bmod~4)$ labels to be incident to the bad verties; the largest of these labels is less than $8n/9$.  Finally, we will assign the $n/9$ smallest $0(\bmod~4)$ labels to be incident to the bad vertices.  Thus, the largest vertex-sum at a bad vertex is less than $3(4n/9)+8n/9 < 3n$.  Hence, no bad vertex will conflict with any vertex in $A$.  

To ensure that no two bad vertices conflict, we assign the labels to the final 1-factor in order of increasing partial sum at the bad vertices.  After we assign all the labels incident to the bad vertices, we assign the remaining labels incident to the good vertices, again in order of increasing partial sum in $B$.  This ensures that no two good vertices conflict.  If the partial sum at a vertex of $A$ is $4n-2$ we call it \textit{small}; otherwise we call it \textit{big}.  After we assign the labels on the final 1-factor, the smallest possible vertex-sum at a big vertex is $(8n-2)+4=8n+2$;  the largest possible sum at a small vertex is $(4n-2)+4n = 8n-2$.  So no small vertex conflicts with a big vertex.  Additionally, all the small vertex-sums are distinct; so are the large vertex-sums.  Thus, the labeling is antimagic.
\end{proof}


\end{document}